\documentclass[10pt]{article}

\usepackage{amsmath}
\usepackage{mathtools}
\usepackage{amssymb}
\usepackage{amsthm}
\usepackage{graphics}
\usepackage[latin1]{inputenc}
\usepackage[T1]{fontenc}

\usepackage{enumerate}

\newtheorem{thm}{Theorem}[section]
\newtheorem{prop}[thm]{Proposition}

\newtheorem{lemma}[thm]{Lemma}
\newtheorem{rem}[thm]{Remark}

\newtheorem{defi}[thm]{Definition}

\newcommand{\R}{\mathbb{R}}             

\DeclarePairedDelimiter\norm{\lVert}{\rVert}

\newcommand{\bu}{{\bf u}}
\newcommand{\bv}{{\bf v}}
\newcommand{\bV}{{\bf V}}
\newcommand{\bw}{{\bf w}}
\newcommand{\be}{{\bf e}}
\newcommand{\bN}{{\bf N}}
\newcommand{\barx}{{\bar x}}

\newcommand{\Section}[1]{\section{#1} \setcounter{equation}{0}}

\begin{document}

\title{Ramified local isometric embeddings\\ of singular Riemannian metrics}
\author{Alberto Enciso $^{\,1}$ 
and 
Niky Kamran  $^{\,2}$\\  
$^1$  \small Instituto de Ciencias Matem\'aticas, Consejo Superior de Investigaciones Cient\'ificas,\\
 \small 28049  Madrid, Spain. \\
\small Email:  aenciso@icmat.es\\
$^2$ \small Department of Mathematics and Statistics, McGill University,\\ \small  Montreal, QC, H3A 0B9, Canada. \\
\small Email: niky.kamran@mcgill.ca \\}





\maketitle


\begin{abstract}

  In this paper, we are concerned with the existence of local isometric embeddings into Euclidean space for analytic Riemannian metrics $g$, defined on a domain~$U\subset \R^n$, which are singular in the sense that the determinant of the metric tensor is allowed to vanish at an isolated  point (say the origin). Specifically, we show that, under suitable technical assumptions, there exists a local analytic isometric embedding $\bu$ from $(U',\Pi^*g)$ into Euclidean space $\mathbb{E}^{(n^2+3n-4)/2}$, where $\Pi:U' \to U\backslash\{0\}$ is a finite Riemannian branched cover of a deleted neighborhood of the origin. Our result can thus be thought of as a generalization of the classical Cartan-Janet Theorem to the singular setting in which the metric tensor is degenerate at an isolated point. Our proof uses Leray's ramified Cauchy-Kovalevskaya Theorem for analytic differential systems, in the form obtained by Choquet-Bruhat for non-linear systems.


\vspace{0.5cm}




\end{abstract}

\tableofcontents


\Section{Introduction} \label{Intro}
A classical theorem of Cartan \cite{C} and Janet \cite{Jan} states that any $n$-dimensional analytic Riemannian manifold $(M,g)$ can be locally isometrically embedded by means of an analytic mapping in $(n(n+1)/2)$-dimensional Euclidean space ${\mathbb{E}}^{n(n+1)/2}$. Cartan's proof rests on the use of the Cartan-K\"ahler Theorem for analytic involutive exterior differential systems, which in turn relies on the iterative application of the Cauchy-Kovalevskaya Theorem, in order to construct the graph of the local isometric embedding one dimension at a time. The hypothesis of analyticity is thus essential in this approach. 

The extension of the Cartan-Janet Theorem to global isometric embeddings of Riemannian manifolds which are only differentiable rather than analytic was obtained much later in the celebrated work of Nash \cite{N}, who took a very different approach based on the development of a powerful open mapping theorem adapted to the problem at hand. Important subsequent improvements of Nash's result are due to Gromov and Rokhlin \cite{Gr}, and G\"unther \cite{Gu}, who gave a much simplified proof of the embedding theorem, including sharper bounds on the codimension. We refer for example to \cite{HH} for an account of these results and a discussion of the bibliography. 

Our purpose in this paper is to consider a much simpler kind of extension of the Cartan-Janet Theorem, one which is still local in nature and in which the metric is still assumed to be analytic, but where the Riemannian metric is allowed to be singular by letting the determinant of the metric tensor admit a simple zero at an isolated point, taken for example at the origin of a local coordinate system. In the framework developed by Cartan and Janet, this corresponds to having the iterative step in the construction of the graph of the embedding being now governed by a Cauchy-Kovalevskaya system for which the Cauchy data are characteristic at a point on the hypersurface supporting the Cauchy data. The Cauchy-Kovalevskaya Theorem is therefore not applicable in this situation.

Cauchy problems of this type have however been studied by Leray and collaborators in an extensive series of papers dealing with systems of Cauchy-Kovalevskaya type in which the Cauchy data are allowed to be characteristic along submanifolds contained in the hypersurface supporting the initial data (see for example~\cite{L},\cite{GKL}, for linear systems, and Choquet-Bruhat \cite{CB} for the non-linear case). The main content of these works, as far as the applications we have in mind in the present paper are concerned, is that the singularities of the Cauchy problem corresponding to the above characteristic submanifolds give rise to solutions of the Cauchy problem which are multi-valued, with a ramification locus whose support can be described geometrically in terms of the conoids generated by the bicharacteristic curves emanating from the characteristic points. Furthermore, provided that the singularities of the Cauchy data are suitably ``non-degenerate'' in a given technical sense, the solution can be uniformized by an analytic map, obtained by solving the Cauchy problem for an auxiliary Hamilton-Jacobi equation whose Hamiltonian is determined by the principal part of the operator. The uniformizing map so obtained then gives rise to analytic solutions of the differential system which are defined in a branched covering of the original domain.

These results strongly suggest that for the singular metrics being considered in this paper, Leray's uniformization theorem might be used in a suitable way instead of the Cauchy-Kovalevskaya Theorem to ensure the key iterative step which consists in the extension of local isometric embeddings. This is precisely the strategy that we apply in our paper. In order to state our result, let us first introduce the class of singular metrics that we can handle through this approach, which we will henceforth refer to as {\em admissible}\/. We find the most transparent way to do so is through a normal form. Here we write points in~$\R^n$ as $x=(x',x_n)$ with $x'\in\R^{n-1}$, the Greek indices range from 1 to~$n$, while the Latin indices range from 1 to $n-1$:

\begin{defi}\label{D.admissible}
We say that a Riemannian metric $g$ defined on a domain $U\ni 0$ of $\mathbb{R}^n$ has an admissible singularity at the origin if it is of the form 
\begin{equation}\label{singmetricexplicit}
g=(\norm{x'}^{2}+x_{n}^{2l})\,F_{0}(x)\,dx_{n}^{2}+g_{kl}\,dx_{k}\,dx_{l}\,,
\end{equation}
where $l\geq 1$ is an integer, where $F$ and $g_{jk}$ are analytic with $F_{0}(0)>0$, where the quadratic form $(g_{jk})$ is positive definite, and where 
\begin{equation}\label{normderiv}
\partial_{n}g_{jk}(x',0)=O(\norm{x'}^{2})\,.
\end{equation}
\end{defi}

In Section~\ref{SingularSection} we shall motivate this normal form, which describes metrics which are degenerate (in a fairly mild sense) in exactly one direction. We will show (see Proposition~\ref{P.bs}) that the assumption that there are no cross terms (i.e., terms in $dx_j\, dx_n$) in the metric is not essential, as even in this singular case one can always put metrics in this form through a suitable analytic change of coordinates. In contrast, the technical assumption about $\partial_n g_{jk}$ will be important.

We are now ready to state our main result:

\begin{thm}\label{T.main}
  Let $g$ be an analytic Riemannian metric defined on a domain $U\subset\R^n$ with an admissible singularity at the origin. Then there exists a local analytic isometric embedding $\bu: (U',\Pi^*g)\to \mathbb{E}^{(n^2+3n-4)/2}$, where $\Pi:U'\to U\backslash\{0\}$ is a finite Riemannian branched cover of~$(U\backslash\{0\},g)$.
\end{thm}

We observe from the statement of Theorem~\ref{T.main} that the codimension of the local isometric embedding that we obtain for the class of singular metrics being considered is larger than the one given the classical Cartan-Janet Theorem. This can be understood from the fact that singularity of the Riemannian metric $g$ at the origin requires more dimensions in the ambient space in order to construct admissible initial data for the key iterative step in which we apply Leray's Theorem. Put differently, we can think of Leray's Theorem as providing a systematic way of resolving ``nondegenerate'' singularities that appear in the Cauchy problem for a Cauchy-Kovalevskaya type system by passing to a branched covering whose geometry is determined by the principal part of the operator. 

Our paper is organized as follows. In Section \ref{PDESection}, we set up the notation and derive in Proposition \ref{equivalence} the constrained Cauchy problem for the system of PDEs that governs the extension of local isometric embeddings in the case in which the coefficient $g_{nn}$ of the metric (\ref{metric}) has not been normalized. This is only a slight variation of the classical set-up of the extension problem as given in \cite{Jac} and \cite{HH}, but it is needed in order to conveniently encode in terms of the Cauchy data for the extension problem the singularity assumption that the determinant of the metric tensor has an isolated zero at the origin corresponding to an admissible singularity. In Section \ref{SingularSection}, we show that in the singular case in which the coefficient $g_{nn}$ of the metric tensor has an isolated zero at the origin, one can nevertheless still put the metric in a normal form without cross terms by means of an analytic change of coordinates. This result, which provides the starting point of the extension problem for local isometric embeddings, is the object of Proposition~\ref{P.bs}, which in turn motivates the definition of the class of admissible singular metrics that we introduced above. Section \ref{RamifiedEmbedSection}, which is the core of the paper, contains both the key technical result needed in order to use Choquet-Bruhat's generalization of Leray's theorem for the construction of ramified extensions, namely the proof of Proposition~\ref{P.initial} giving the existence of Cauchy data which satisfy the appropriate rank conditions on the complement of the origin in the hypersurface supporting the Cauchy data, and the proof of our main result, that is Theorem~\ref{T.main}. The Appendix is devoted to giving a reasonably self-contained summary of the essentials of Leray's and Choquet-Bruhat's theorems on the ramification and uniformization of the solution of the Cauchy problem for analytic differential systems near a characteristic point. In particular, the non-degeneracy conditions required for the uniformization to correspond to algebroid singularities of the solution along the ramification locus are stated explicitly.

We conclude by remarking that the results of Leray and Choquet-Bruhat on the ramification and uniformization of the solution of the Cauchy problem in general singular settings in which the Cauchy data are allowed to be characteristic on submanifolds of the hypersurface supporting the data will probably be relevant in singular geometric situations which are broader in scope than the one considered this paper. Ultimately, one might hope for a generalized Cartan-K\"ahler Theorem involving a more general notion of involutivity, producing ramified integral manifolds and corresponding uniformization maps. We think that the development of such a theorem, while entailing a significant amount of technical content, would be a goal worth pursuing with specific geometric applications in mind.

\section{A PDE system for the extension of local isometric embeddings without a normal form}\label{PDESection}

Let $U$ be an open neighbourhood of the origin in $\mathbb{R}^{n}$, with local coordinates $x=(x_{\alpha})_{1\leq \alpha \leq n}$. We shall also use the notation $x=(x',x_n)$ for the local coordinates, where $x'=(x_{k})_{1\leq k\leq n-1}$. On $U$, we consider for now a smooth Riemannian metric of the form
\begin{equation}\label{metric}
g=g_{nn}(x)\, dx_{n}^{2}+g_{kl}(x)\, dx_{k}\, dx_{l}\,.
\end{equation}
where the sum of $k,l$ ranges from $1$ to $n-1$, but where the coefficients $g_{nn}$ and $g_{kl}$ are allowed to be functions of all the local coordinates $x=(x_{\alpha}), 1\leq \alpha \leq n$.

Given the positive-definiteness and smoothness of $g$ in $U$, we may with no loss of generality choose normal coordinates in a neighbourhood of the origin in $U$ in which the normalization condition $g_{nn}=1$ is satisfied.  However we shall be interested in Section \ref{RamifiedEmbedSection} in constructing local isometric embeddings in the singular case in which $g_{nn}$ has an isolated zero of a certain specific type at the origin, so that such a coordinate system will in general fail to exist. The precise conditions that we will to impose on the order of vanishing of $g_{nn}$ at the origin are stated as Definition~1.1. In anticipation of this, we shall keep the component $g_{nn}$ of the metric $g$ un-normalized while still working initially under the assumption that $g$ is a non-singular smooth Riemannian metric in $U$, so that $g_{nn}\neq 0$ in $U$.

A smooth map $\bu: U \to \mathbb{E}^{N}$ of rank $n$ will produce a local isometric embedding of $(U,g)$ into the $N$-dimensional  Eucliden space $ \mathbb{E}^{N}$ if and only if it satisfies the system of  ${n(n+1)}/2$ first-order PDEs given by 
\begin{align}
\norm{\partial_{n}{\bu}}^{2}&=g_{nn}\,,\label{I.1}\\ 
\partial_{k}{\bu}\cdot\partial_{n}{\bu}&=0\,,\label{I.2}\\
\partial_{k}{\bu}\cdot\partial_{l}{\bu}&=g_{kl}\,,\label{I.3}
\end{align}
where we use a dot to denote the Euclidean inner product in the ambient Euclidean space $\mathbb{E}^{N}$. We shall consider Cauchy data for the system (\ref{I.1}), (\ref{I.2}), (\ref{I.3}) along the hypersurface $x_n=0$, given by smooth maps $\bu_{0},\bu_{1}$,
\begin{equation}\label{initial}
\bu |_{x_n=0}=\bu_{0}\,,\quad \partial_{n}\bu |_{x_n=0}=\bu_{1}\,,
\end{equation}  
where the data are constrained in view of (\ref{I.1}),  (\ref{I.2}) and  (\ref{I.3}) by
\begin{align}
\norm{{\bu}_{1}}^{2}&=g_{nn}(\cdot,0)\,,\label{7}\\ 
\partial_{k}{\bu}_{0}\cdot{\bu_{1}}&=0\,,\label{8}\\
\partial_{k} \bu_{0}\cdot\partial_{l} \bu_{0}&=g_{kl}(\cdot,0)\,,\label{9}
\end{align}

Following the classical procedure employed in the proof of the Cartan-Janet Theorem, (see \cite{Jac}, \cite{HH}), we differentiate the equations (\ref{I.1}), (\ref{I.2}), (\ref{I.3}) with respect to $x_n$ to obtain the system of ${n(n+1)}/2$ second-order PDEs given by
\begin{align}
\partial_{n}\bu\cdot\partial_{nn}\bu&=\frac{1}{2}\partial_{n}g_{nn}\,,\label{II.1}\\
\partial_{k}\bu\cdot\partial_{nn}\bu&=-\frac{1}{2}\partial_{k}g_{nn}\,,\label{II.2}\\
\partial_{kl}\bu\cdot\partial_{nn}\bu&=\partial_{kn}\bu\cdot\partial_{ln}\bu-\frac{1}{2}\partial_{nn}g_{kl}-\frac{1}{2}\partial_{kl}g_{nn}\,.  \label{II.3}
\end{align}
It is straightforward to show following a procedure identical to the one used in the case when $g_{nn}=1$ ((see for example \cite{HH}),) that the Cauchy data $\bu_{0}\,,\bu_{1}$ must satisfy, in addition to (\ref{7}), (\ref{8}), (\ref{9}), the constraint 
\begin{equation}\label{10}
\partial_{kl} \bu_{0}\cdot\bu_{1}=-\frac{1}{2}\partial_{n}g_{kl}(\cdot,0)\,.
\end{equation}
One then obtains, in complete analogy with the case in which the $g_{nn}=1$, the following:
\begin{prop}\label{equivalence}
Consider a smooth Riemannian metric (\ref{metric}). The system of first-order PDEs (\ref{I.1}), (\ref{I.2}), (\ref{I.3}) governing local isometric embeddings of $(U,g)$ into ${\mathbb E}^{N}$, with initial data (\ref{initial}) constrained by (\ref{7}) to (\ref{9}), is equivalent to the system of second-order PDEs (\ref{II.1}), (\ref{II.2}), (\ref{II.3}) with initial data (\ref{initial}) constrained by (\ref{7}) to (\ref{9}) and (\ref{10}).
\end{prop}

We conclude this section by recalling that in the hypotheses of the classical local isometric embedding theorem of Cartan-Janet \cite{C,Jan,Jac}, one needs to assume that the metric $g$ is analytic in $U$, given that the proof consists in constructing the local isometric embedding by induction on $n$, applying the Cauchy-Kovalevskaya Theorem at each stage to the system (\ref{II.1}), (\ref{II.2}), (\ref{II.3}) with initial data satisfying (\ref{7}) to (\ref{9}) and (\ref{10}). The Cartan-Janet Theorem then ensures the existence of a local isometric embedding when $N=n(n+1)/2$. For future reference, we summarize the iterative step in the proof as a proposition, which follows directly from the Cauchy-Kovalevskaya Theorem by putting the system of PDEs (\ref{II.1}), (\ref{II.2}), (\ref{II.3}) in Cauchy-Kovalevskaya form by solving for $\partial_{nn}\bu$, using a rank hypothesis on the initial data, {see \cite{Jac}},\cite{HH}:
\begin{prop}\label{iteration}
Consider an analytic Riemannian metric (\ref{metric}). The equivalent system of second-order PDEs (\ref{II.1}), (\ref{II.2}), (\ref{II.3}) with initial data $\bu_{0},\bu_{1}$  constrained by (\ref{7}) to (\ref{9}) and (\ref{10}) governing local isometric embeddings of $(U,g)$ into ${\mathbb E}^{N}$ admits a unique local analytic solution $\bu$ if the set $\{\partial_{k}\bu_{0},\bu_{1}, \partial_{kl}\bu_{0}\,,\,1\leq k,l \leq n-1\}$ is linearly independent at every point of the initial hypersurface $x_{n}=0$.
\end{prop}

As remarked earlier, our goal in Section \ref{RamifiedEmbedSection} will be to solve the Cauchy problem for the system (\ref{I.1}), (\ref{I.2}), (\ref{I.3}) in the singular setting in which $g_{nn}$ has a zero of a certain type at the origin. Since our proof will rely on Choquet-Bruhat's version~\cite{CB} for non-linear systems of Leray's extension of the Cauchy-Kovalevskaya Theorem to the case in which the Cauchy data admit (non-exceptional) characteristic points \cite{L},,\cite{GKL}, we shall from now on assume that the functions $g_{nn}$ and $g_{kl}\,,1\leq k,l\leq n-1$ are analytic in $U$. The precise hypotheses on the behaviour of $g_{nn}$ are the ones given in Definition 1.1.
\section{Singular metrics} \label{SingularSection}

Our goal for this section of our paper is to introduce the class of singular metrics for which we shall prove the existence of ramified local isometric embeddings in Section \ref{RamifiedEmbedSection}. 

We begin by showing that the normal form (\ref{metric}), in which the components $g_{nj}$ of the metric tensor are identically zero, can be achieved by a suitable analytic local diffeomorphism for a class singular metrics which include the singular metrics covered by Definition~1.1 as a special case.  More precisely, we prove the following result:

\begin{prop}\label{P.bs}
Consider on a domain $U\subset\R^n$ a singular analytic Riemannian metric 
\begin{equation}\label{generalmetric}
g=g_{\alpha \beta}(x)\,dx_{\alpha}\,dx_{\beta}=g_{nn}(x)\,dx_n^{2}+2b_{j}(x)\,dx_j\,dx_n+g_{jk}(x)\,dx_j\,dx_k\,,
\end{equation}
where by singular we mean that $g_{nn}$ has an isolated zero at the origin $0\in U$ and $\det(g_{kl})\neq 0$ in $U$. Then there exists an analytic local diffeomorphism $f$ of the form 
\begin{equation}\label{diffeo}
x_n={\bar x}_n\,,\quad x_{j}={\bar x}_{j}+f_{j}({\bar x})\,,
\end{equation}
such that the components ${\bar b}_{j}$ of the transformed metric ${\bar g}:=f^{*}g$ are identically zero, that is $\bar g$ takes the form 
\begin{equation}\label{normalmetric}
{\bar g}={\bar g}_{nn}({\bar x})\,d{\bar x}_{n}^{2}+{\bar g}_{kl}(\bar x)\,d{\bar x}_{k}\,d{\bar x}_{l}\,.
\end{equation}
\end{prop}
\begin{proof}
The first step in the proof is to show that the coefficients $b_j$ in (\ref{generalmetric}) must vanish at the origin as a consequence of our hypothesis that $g_{nn}$ has an isolated zero at the origin. Let 
\[
V=V'\partial_{x'}+V_{n}\partial_n\,,
\]
denote a non-zero tangent vector at a point $x\in U$. We have, at the point $x$,
\[
0<g(V,V)=\norm{V'}_{g'}^{2}+2 V_nV'\cdot b+g_{nn}(V_{n})^{2}\,,
\]
where $g':=(g_{kl})$ is the $(n-1)\times (n-1)$ sub-matrix of $g$ corresponding to the range $1\leq k,l \leq n-1$ and the dot stands for Euclidean scalar product. Defining $\frak {b}\in \mathbb{R}^{n-1}$ by $b=:g'\frak{b}$, the above inequality then reads
\[
0<g(V,V)= \norm{V'}_{g'}^{2}+2 V_n\, g'( V',\frak{b})+g_{nn}(V_{n})^{2}\,.
\]
The worst possible case for this inequality occurs when $V'=\lambda\, \frak{b}$,where $\lambda \in \mathbb{R}$, in which case the condition $0<g(V,V)$ reduces to 
\[ 
0<\lambda^{2}\norm{\frak{b}}_{g'}^{2}+2\lambda V_{n}\norm{\frak{b}}_{g'}^{2}+g_{nn}(V_{n})^{2}\,.
\]
This inequality will hold if an only if 
\[
\norm{\frak{b}}_{g'}^{2}g_{nn}(V_{n})^{2}-\|\frak{b}\|_{g'}^{4}(V_{n})^{2}>0\,,
\]
for all $V_{n}\neq 0$, or equivalently 
\[
\norm{\frak{b}}_{g'}^{2}<g_{nn}\,,
\]
which establishes the first step.

Next we apply a local diffeomorphism of the form (\ref{diffeo}) to the metric (\ref{generalmetric}) and determine the conditions that the functions $f_j({\bar x})$ must satisfy in order for transformed metric $f^{*}g$ to take the form (\ref{normalmetric}) in which the coefficients ${\bar b}_j$ of the cross terms in the metric 
\[
f^{*}g={\bar g}_{\alpha \beta}({\bar x})\,d{\bar x}_{\alpha}\,d{\bar x}_{\beta}={\bar g}_{nn}({\bar x})\,d{\bar x}_n^{2}+2{\bar b}_{j}({\bar x})\,d{\bar x}_j\,d{\bar x}_n+{\bar g}_{jk}(x)\,d{\bar x}_j\,d{\bar x}_k\,,
\]
are identically zero. A straightforward calculation gives that ${\bar b}_{j}=0$ if and only if 
\begin{equation}\label{killb}
b_{k}\partial_{n}f_{k}+g_{jk}\partial_{l}f_{j}\partial_{n}f_{k}=-b_{l}\,,
\end{equation}
where all the partial derivatives are taken with respect to the barred coordinates $({\bar x}',{\bar x}_n)$. We now choose an invertible matrix $A=(A_{jk})$ of small norm, say $\norm{A}<\varepsilon$, and define $n-1$ linear functions ${\tilde f}_{j}(\barx)$ by
\[
{\tilde f}_{j}(\barx)=(g^{-1})_{jm}A_{lm}\barx_{l}\/.
\]
Define next $G(\partial' f):=M^{-1}(\partial' f)$, where
\[
M_{lk}(\partial'f):=b_{k}+\partial_{l}f_{j}g_{jk}\,.
\]
We claim that the matrix-valued function $G(\partial' f)$ is well defined for $f_{j}$ in a $C^{1}$-small neighbourhood of ${\tilde f}_{j}$ and $x$ in a small neighbourhood of the origin. Indeed, by definition of $\tilde f$ and using the fact that $b_{j}(0)=0$, we have 
\[
M_{lk}(\partial'f)(\barx)=O(\norm{\bar x})+(g^{-1})_{jm}A_{lm}g_{jk}+O(\norm{f-{\tilde f}}_{C^{1}})=A_{lk}+O(\norm{\bar x}+\norm{f-{\tilde f}}_{C^{1}})\,.
\]
Therefore we can write the system of equations (\ref{killb}) to be solved as 
\begin{equation}\label{killbnorm}
\partial_{n}f_{k}=-G_{kl}(\partial'f)b_{l}\,,
\end{equation}
and take as initial data 
\begin{equation}\label{killbinit}
f_{k}|_{{\barx}_{n}=0}={\tilde f}_{k}=(g^{-1})_{km}A_{lm}{\barx}_{l}\,.
\end{equation}
We now apply the Cauchy-Kovalevskaya Theorem to (\ref{killbnorm}) with initial data given by (\ref{killbinit}). By choosing $\varepsilon$ small enough, we can ensure that the map (\ref{diffeo}) obtained by solving the system (\ref{killbnorm}) is a local diffeomorphism, which proves our claim.
\end{proof}
Now that the normal form 
\begin{equation}\label{singmetric}
g=g_{nn}(x)\,dx_{n}^{2}+g_{kl}(x)\,dx_{k}\,dx_{l}\,.
\end{equation}
has been established for singular metrics such that $g_{nn}$ has an isolated zero at the origin $0\in U$, we need to make further assumptions about the leading order behaviour of $g_{nn}$ and of the normal derivative of $g_{ij}$ at the origin in order to be able to apply the results of Leray's ramified Cauchy-Kovalevskaya theorem (which are recalled in Appendix \ref{LerayTheorem}). Specifically, we require that the origin be an isolated zero of $g_{nn}$, and that both the partial Hessian $(\partial_{jk}g_{nn}(0))_{1\leq j,k\leq n-1}$ and the matrix $(g_{jk}(0)) _{1\leq j,k\leq n-1}$ be positive definite. 
Transforming the partial Hessian at~0 into the identity matrix with a local change of coordinates and employing the division property of analytic functions, this leads to the definition of admissible singularities presented in the Introduction as Definition~1.1.
This definition is the starting point of the analysis of the ramified local isometric embedding problem which we shall work out in the next section.

\section{Ramified local isometric embeddings}\label{RamifiedEmbedSection}

In analogy with the construction of Cauchy data for the iterated sequence of Cauchy-Kovalevskaya problems that come up in the proof of the classical Cartan-Janet Theorem (see Proposition \ref{iteration}), our next goal for this section is to show that there exist analytic Cauchy data $\bu_{0},\, \bu_{1}$ for the local isometric embedding problem of the class of singular metrics with an admissible singularity at the origin, which satisfy the constraints (\ref{7}) to (\ref{10}) and which are such that the Cauchy problem for the system (\ref{II.1}) to (\ref{II.3}) admits an isolated \emph{non-exceptional characteristic point} in the sense of Definition \ref{nondeg} at the origin $0\in U$. This turns out to be significantly more delicate than the construction that appears in the proof of the classical Cartan-Janet Theorem. 

We begin by noting that letting 
\[
F(x'):=F_{0}(x',0),\quad{\bar g}_{ij}(x'):=g_{ij}(x',0),\quad h_{ij}(x'):=-\frac{1}{2}\partial_{n}g_{ij}(x',0)\,,
\]
the constraints (\ref{7}) to (\ref{10}) read
\begin{align}
\|\bu_{1}\|^{2}&=\norm{x'}^{2}F\,.\label{C1}\\
\partial_{i}\bu_{0}\cdot\bu_{1}&=0\,,\label{C2}\\
\partial_{i}\bu_{0}\cdot\partial_{j}\bu_{0}&={\bar g}_{ij}\,,\label{C3}\\
\partial_{ij}\bu_{0}\cdot\bu_{1}&=h_{ij}\,,\label{C4}
\end{align}
and the assumption that the origin should be an isolated characteristic point for the system of PDEs (\ref{II.1}),(\ref{II.2}), (\ref{II.3}) implies that the vectors $\partial_{i}\bu_{0}, \bu_{1}, \partial_{ij}\bu_{0}$ should be linearly independent at the image of every point $x'\neq 0$ in the domain of definition of $\bu_{0}$ and $\bu_{1}$.

Recall first that for any analytic (non-degenerate) metric ${\hat g}_{ij}(x')$ on the hypersurface $x_n=0$, we know from the Cartan-Janet Theorem that there exits a local isometric embedding from $\Sigma \subset \{x_{n}=0\} \to \mathbb{E}^{N'}$, where $N'=n(n-1)/2$, meaning that there exists a local isometric embedding $\bv$ which satisfies 
\begin{equation}\label{ghat}
\partial_{i}\bv\cdot\partial_{j}\bv={\hat g}_{ij}\,,
\end{equation}
where we may assume with no loss of generality that $\partial_{a}\bv,\partial_{n-1}\bv,\partial_{ab}\bv,\,1\leq a,b \leq n-2$, are linearly independent at every $x'\in \Sigma$. Consider now the embedding ${\bw}: \Sigma \to \mathbb{E}^{N-1}= \mathbb{E}^{N'}\times  \mathbb{E}^{n-1}$ defined by
\[
\bw:=(\bv,\bV)\,,
\]
with $\bV=(V_{1},\dots,V_{n-1})$ defined by 
\[
V_{a}:=\epsilon^{5}\sin\frac{x_{n-1}}{\epsilon^{4}}\,\sin\frac{x_{a}}{\epsilon^{2}}\,,\quad V_{n-1}:=-\epsilon^{5}\cos\frac{x_{n-1}}{\epsilon^{4}}\,,
\]
where $1\leq a \leq n-2$. It is then straightforward to verify that 
\begin{align*}
\partial_{a}\bw &= \partial_{a}\bv +O(\epsilon^{3})\,,\\
\partial_{n-1}\bw &= \partial_{n-1}\bv +O(\epsilon)\,\\
\partial_{ab}\bw &= \partial_{ab}\bv +O(\epsilon)\,\\
\partial_{n-1,a}\bw &= \partial_{n-1,a}\bv +\frac{1}{\epsilon}\cos\frac{x_{n-1}}{\epsilon^{4}}\,\cos\frac{x_{a}}{\epsilon^{2}}{\bf E}_{a}\,\\
\partial_{n-1,n-1}\bw &= \partial_{n-1,n-1}\bv +\frac{1}{\epsilon^{3}}\bigg(\cos\frac{x_{n-1}}{\epsilon^{4}}{\bf E}_{n-1}-\sin\frac{x_{n-1}}{\epsilon^{4}}\sum_{a=1}^{n-2}\sin\frac{x_a}{\epsilon^{2}}\,{\bf E}_{a}\bigg)\,,
\end{align*}
where ${\bf{E}}_i$ are unit vectors in the $x_i$ direction in $\mathbb{E}^{n}$.  So if we take $\epsilon$ very small and choose $x'$ with $\norm{x'}<\epsilon^{5}$, we may conclude that $\partial_{j}\bw$ and $\partial_{jk}\bw$ are linearly independent at every point of their domain of definition in $\mathbb{R}^{N-1}$. We also have 
\[
\partial_{i}\bw\cdot\partial_{j}\bw=\partial_{i}\bv\cdot\partial_{j}\bv+\partial_{i}\bV\cdot\partial_{j}\bV\,,
\]
where 
\begin{equation}\label{estimate}
|\partial_{i}\bV\cdot \partial_{j}\bV|<C\epsilon^{2}\,,
\end{equation}
for some positive constant $C$. 

If we set now
\[
{\hat g}_{ij}:={\bar g}_{ij}-\partial_{i}\bV\cdot\partial_{j}\bV\,,
\]
then this metric is now positive-definite the origin as a consequence of the estimate (\ref{estimate}), so we may apply the Cartan-Janet Theorem as we did above to ${\hat g}_{ij}$ to conclude using (\ref{ghat}) that the map $\bw$ satisfies
\[
\partial_{i}\bw\cdot\partial_{j}\bw={\bar{g}}_{ij}\,.
\] 
Next we consider the embedding $\bu_{0}:\Sigma \to
\mathbb{E}^{N-1}\times \mathbb{E}^{n-1}=\mathbb{E}^{N+n-2}$ given by
\begin{equation}\label{defu0}
\bu_{0}:=(\bw,0)\,.
\end{equation}
The tangent space $T_{\bu_{0}(x')}{\bar{\Sigma}}$ to ${\bar \Sigma}=\bu_{0}(\Sigma)\subset \mathbb{E}^{N+n-2}$ at $\bu_{0}(x')$ is the $(n-1)$-dimensional subspace given by the linear span of the vectors $\partial_{j}\bw,\,1\leq j \leq n-1$, that is 
\begin{equation}\label{tangent}
 T_{\bu_{0}(x')}{\bar{\Sigma}}=\langle\partial_{j}\bw : 1\leq j \leq n-1\rangle\,.
\end{equation}
We also have
\[
\langle\partial_{j}\bw,\partial_{jk}\bw : 1\leq j,k \leq
n-1\rangle=\langle\partial_{j}\bw,\bN_{r} : 1\leq j \leq n-1,\,1\leq r \leq n(n-1)/2\rangle\,,
\]
where $\{\bN_{r},\, 1\leq r \leq n(n-1)/2\}$ is a linearly independent of unit normal vectors in $\mathbb{E}^{N-1}\subset \mathbb{E}^{N+n-2}$. 

Denoting by $\{\be_{j} : 1\leq j\leq n-1\}$ an orthonormal basis of
the $\mathbb{E}^{n-1}$ factor in
$\mathbb{E}^{N+n-2}=\mathbb{E}^{N-1}\times \mathbb{E}^{n-1}$, the
normal space $N_{\bu_{0}(x')}{\bar{\Sigma}}$ to ${\bar
  \Sigma}=\bu_{0}(\Sigma)\subset \mathbb{E}^{N+n-2}$ at $\bu_{0}(x')$
is, in view of the linear independence of set of vectors
$\{\partial_{j}\bw, \partial_{ab}\bw : 1\leq j \leq n-1,\,1\leq a,b \leq n-2\}$ (which follows from Proposition \ref{iteration}), given by the $(N-1)$-dimensional subspace  
\begin{equation}\label{normal}
N_{\bu_{0}(x')}{\bar{\Sigma}}=\langle\bN_{r},\be_{j} : 1\leq r \leq n(n-1)/2, 1\leq j\leq n-1 \rangle\,.
\end{equation}
Consequently, there exists a unique vector field ${\bf N}$ along ${\bar \Sigma}$ of the form 
\[
{\bf N}=\sum_{r=1}^{n(n-1)/2}\alpha_{r}{\bf N}_{r}\,,
\]
such that 
\[
{\bf N}\cdot\partial_{jk}\bw=h_{ij}\,.
\]
Note that the hypothesis (\ref{normderiv}) on $g$ can be rewritten as $h_{ij}=O(\|x'\|^2)$, so we immediately infer that
\begin{equation}\label{normN}
\norm{{\bf N}}=O(\norm{x}^{2})\,.
\end{equation}
We now set
\begin{equation}\label{defu1}
\bu_{1}:={\bf N}+\sum_{j=1}^{N}x_{j}\,G{\be}_{j}\,,
\end{equation}
where 
\begin{equation}\label{defG}
G:=\bigg(F-\frac{\norm{{\bf N}}^{2}}{\norm{x'}^{2}}\bigg)^{1/2}
\end{equation}
is a real-valued analytic function near $x'=0$ by the bound~\eqref{normN}.
It follows from (\ref{tangent}), (\ref{normal}), (\ref{defu1}) and
(\ref{defG}) that the initial data $\bu_{0}, \bu_{1}$ defined by
(\ref{defu0}), (\ref{defu1}) satisfy the constraints (\ref{C1}),
(\ref{C2}), (\ref{C3}) and (\ref{C4}).

Furthermore, by construction,
the function $\Delta:V\to\R$ defined as
\begin{equation}\label{Delta}
\Delta(x'):=\det (\partial_j \bu_0(x'),\bu_1(x'),\partial_{jk}\bu_0(x'),
\be_a)_{1\leq j,k\leq n-1,\; 2\leq a\leq n-1}
\end{equation}
has a nondegenerate zero at 0. More precisely, it follows from the previous construction that, in a neighborhood of the origin, the function~$\Delta$ is of the form
\[
\Delta(x')=x_1\, \Delta_0(x')
\]
with $\Delta_0(0)\neq0$. To see this, note that the linear independence of the above vectors implies that
\begin{align*}
\Delta(x')&=\det (\partial_j \bu_0(x'),\bu_1(x'),\partial_{jk}\bu_0(x'),
            \be_a)_{1\leq j,k\leq n-1,\; 2\leq a\leq n-1}\\
  &= \det \bigg(\partial_j \bu_0(x'), {\bf N}(x')+\sum_{l=1}^{N}x_{l}\,G(x')\,{\be}_{l},\partial_{jk}\bu_0(x'),
    \be_a\bigg)_{1\leq j,k\leq n-1,\; 2\leq a\leq n-1}\\
  &=\det \big(\partial_j \bu_0(x'),x_1\,G(x')\, {\be}_{1},\partial_{jk}\bu_0(x'),
    \be_a\big)_{1\leq j,k\leq n-1,\; 2\leq a\leq n-1}\\
  &=x_1\,G(x')\, \det \big(\partial_j \bu_0(x'), {\be}_{1},\partial_{jk}\bu_0(x'),
    \be_a\big)_{1\leq j,k\leq n-1,\; 2\leq a\leq n-1}\\
  &=: x_1\, \Delta_0(x')\,,
\end{align*}
where indeed $\Delta_0(0) \neq0$. We have thus proved:

\begin{prop}\label{P.initial}
Consider an analytic metric $g$ on a neighbourhood $U \subset
\mathbb{R}^{n}$ admitting an admissible singularity at the origin,
that is a metric of the form (\ref{singmetricexplicit}) satisfying
(\ref{normderiv}). Then there exist analytic initial data $\bu_{0},
\bu_{1}$ for the system (\ref{II.1}) to (\ref{II.3}), taking values in
$\mathbb{E}^{N+n-2}$, which satisfy the constraints (\ref{C1}) to
(\ref{C4}) and which are such that $\partial_{i}\bu_{0}, \bu_{1},
\partial_{ij}\bu_{0}$ are linearly independent on the complement of
the origin in the initial hypersurface $x_{n}=0$, The
function~$\Delta(x')$ defined in~\eqref{Delta} satisfies $\partial_1\Delta(0)\neq0$.
\end{prop}

Proposition~\ref{P.initial} should be thought of as the key technical result of
the paper. Now, armed with this result, it is now easy to establish our main result, that is Theorem~\ref{T.main}: 

\begin{proof}[Proof of Theorem~\ref{T.main}]
The hypothesis that $\det g$ has a simple zero at the origin ensures
that the tensor~$g$ has only one null direction at~0, so it can be
written in the form~\eqref{generalmetric}. Proposition~\ref{P.bs}
ensures we can get rid of the crossed terms (i.e., those in $dx_j\,
dx_n$) by means of a local change of coordinates. The embedding
problem is then reduced to studying the system of
PDEs~\eqref{II.1}--\eqref{II.3} for a metric~$g$ that is admissible in
the sense of Definition~1.1. Since the system is
underdetermined in that there are fewer equations ($N=n(n+1)/2$) than
unknowns ($N+n-2= (n^2+3n-4)/2$), let us
augment the system by imposing that
\begin{equation}\label{II.last}
\be_a\cdot \partial_{nn}\bu=0\,,\qquad 2\leq a\leq n-1
\end{equation}
where the orthonormal vectors $\{\be_a\}_{a=2}^{n-1}$ are
defined as before.

To construct a solution to the augmented system of PDEs
\eqref{II.1}-\eqref{II.2}-\eqref{II.3}-\eqref{II.last}, we employ
Leray's Cauchy-Kovalevskaya theorem in the form given by Choquet-Bruhat for non-linear systems (see Theorem \ref{T.CB}). Using the notation we introduce there,
consider the Cauchy surface $S:=\{x\in\R^n:x_n=0\}$, corresponding to
the function $s(x):=x_n$. It is not hard to check,
that the function $\mathcal{A}(x,p)$ defined in
the Appendix is an $(N+n-2)\times (N+n-2)$ matrix of the form
\[
  \mathcal{A}(x,p)=p_n^3 \left(\partial_j\bu(x),
                         \partial_n\bu(x),
                         \partial_{jk}\bu(x),
                         \be_a 
               \right)_{1\leq j,k\leq n-1,\; 2\leq a\leq n-1} + \sum_{\alpha} p^\alpha \mathcal{M}_\alpha(x)\,,
             \]
where the sum ranges over the set of multiindices with
$|\alpha|=3$ such that the monomial $p^\alpha$ is different from
$p_n^3$ and $\mathcal{M}_\alpha(x)$ are matrices whose concrete
expressions will not be needed.

With the Cauchy data
\begin{equation}\label{Cauchu}
\bu|_{x_n=0}=\bu_0\,,\qquad \partial_n\bu|_{x_n=0}=\bu_1\,,
\end{equation}
and using the fact that the
gradient of $s(x)=x_n$ points in the $n$-th direction,
we immediately obtain from the previous formula that, on~$S$, the function
$\mathcal{A}_*$ defined in the Appendix is precisely
\[
  \mathcal{A}_*(x')=\Delta(x')\,,
                          \]
                          where the function~$\Delta$ was introduced
                          in~\eqref{Delta}.  In the case where the
                          metric is not singular, this formula is of
                          course very well-known.

                          Proposition~\ref{P.initial}
                          ensures that $\partial_1\Delta(0)\neq0$, so there is a direction tangent to
                          the hyperplane~$S$ such that the
                          corresponding directional derivative
                          of~$\Delta$ at~0 does not vanish. By
                          Remark~\ref{R.nonex}, the origin is then a
                          non-exceptional characteristic point of the system.
                          Choquet-Bruhat's
                          nonlinear extension of Leray's theorem,
                          stated in Theorem~\eqref{T.CB} then shows that,
                          in a small deleted neighborhood of~$0$, the
                          system
                          \eqref{II.1}-\eqref{II.2}-\eqref{II.3}-\eqref{II.last}
                          admits a unique ramified solution with the
                          initial data~\eqref{Cauchu}, and that the
                          singularities of~$\bu$ in a neighborhood
                          of~0 are algebroid. This implies that there is
                          a finite Riemannian cover $U'$
                          of~$U\backslash\{0\}$ as in the statement of
                          the theorem such that $\bu$ defines an
                          embedding $U'\to\mathbb{E}^{N+n-2}$.
\end{proof}

\appendix

\section{Leray's ramified Cauchy-Kovalevskaya Theorem}\label{LerayTheorem}
Our purpose in this appendix is to summarize the essentials of Leray's theorem on the ramification and uniformization of the solution of the Cauchy problem for analytic differential systems near a characteristic point, \cite{L},\cite{ GKL}. 

In order to keep the ideas transparent and so as to avoid unnecessary notational complexities, we shall focus primarily on the case of scalar linear differential operators treated originally by Leray, following the lucid exposition given by Garding in \cite{G}. The proof of Leray's theorem (which we shall refer to as the uniformization theorem for short) is presented very clearly in \cite{G} and will thus be omitted from this appendix for reasons of space. The extension of the uniformization theorem to the case of non-linear analytic systems, which is the result we are actually using in connection with the local isometric embedding problem for metrics with an admissible singularity at the origin, was established by Choquet-Bruhat in \cite{CB}, a few years after the publication of the foundational papers of Leray and his collaborators \cite{L}, \cite{GKL}. We should remark at this stage that the statement of Choquet-Bruhat's uniformization theorem and its proof follow, with a number of non-trivial modifications, the main steps appearing already in \cite{L} and \cite{GKL} for the linear case. These modifications notably entail a significant increase in notational complexity due to the differentiation of the original system (an important step which is required to put the system in an equivalent quasi-linear form), and the introduction thereby of higher-order derivatives. Therefore, instead of presenting the statement of Choquet-Bruhat's theorem in full detail after having reviewed Leray's theorem for scalar linear differential operators, we shall focus on those features which are sufficient for the application we make in our paper to the local isometric embedding problem for metrics with an admissible singularity at the origin, keeping in mind that the linear case already contains the key ingredients, ideas and techniques that are at the basis of the uniformization theorem for non-linear systems.

We thus begin with the scalar linear case, and consider on an subset $V$ of $\mathbb{R}^{l}$ with local coordinates $x=(x_{1}\,\ldots,x_{l})$ an $m$-th order linear differential operator 
\begin{equation}\label{DefA}
A=a(x,\partial_x)=\sum_{|\alpha|\leq m}a_{\alpha}(x)\partial_{x}^{\alpha}\,,
\end{equation}
with analytic coefficients, where we use the standard multi-index notation for derivatives. We shall be interested in the Cauchy problem for the PDE
\begin{equation}\label{scalarop}
a(x,\partial_{x})u(x)=v(x)\,,
\end{equation}
where $v$ is analytic in $V$. The Cauchy problem will consist in prescribing the values of an analytic function $w$ and its derivatives of order $0\leq k\leq m-1$ on an analytic hypersurface $S\subset V$ given as the zero set $s(x)=0$ of an analytic function $s$, 
\begin{equation}\label{scalardata}
u(x)-w(x)=O(s(x)^{m})\,.
\end{equation}
The principal part $G$ of $A$, defined by 
\begin{equation}\label{principal}
G=g(x,\partial_x):=\sum_{|\alpha| = m}a_{\alpha}(x)\partial_{x}^{\alpha}\,,
\end{equation}
will play an important role in the analysis of the ramified Cauchy problem. It defines on the cotangent bundle $T^{*}V$ an analytic real-valued function $g$ given in standard bundle coordinates $(x,p)=(x_{1},\ldots,x_{l},p_{1},\ldots,p_{l})$ by 
\begin{equation}\label{Hamiltonian}
g(x,p)=\sum_{|\alpha| = m}a_{\alpha}(x)p^{\alpha}\,.
\end{equation}
The function $g(x,p)$ is thus homogeneous of degree $m$ in the fiber coordinates $p=(p_{1},\ldots,p_{l})$, and is invariant under local diffeomorphisms of $T^{*}V$ which are lifts of local diffeomorphisms of V.

It is of course well known that the classical Cauchy-Kovalevskaya Theorem (see for example \cite{T}) guarantees the existence of a unique analytic solution to the Cauchy problem (\ref{scalarop}),(\ref{scalardata}) for data that are non-characteristic. It is useful at this stage to recall how the condition for a point $x\in S$ to be characteristic for the Cauchy problem is expressed in terms of the Cauchy data $w$, the hypersurface $S$ and the principal part $G$ of $A$:
\begin{defi}\label{chardefi}
We say that a point $x\in S$ is characteristic for the Cauchy problem (\ref{scalarop}),(\ref{scalardata}) if
\begin{equation}\label{char}
g(x,\partial_xs(x))=0\,.
\end{equation}
The subset of characteristic points $x\in S$ will be denoted by $C$. 
\end{defi}
Leray's extension of the Cauchy-Kovalevskaya Theorem is precisely concerned with this case where the data are allowed to be characteristic on a non-empty subset of the initial hypersurface $S$, in which case in which the classical Cauchy-Kovalevskaya Theorem will fail to apply. The main content of Leray's theorem, which is stated below as Theorem \ref{Leraythm}, is first that the Cauchy problem (\ref{scalarop}),(\ref{scalardata}) will have a solution $u$ that is analytic away from a ramification locus that can be described geometrically in terms of the flow of a Hamiltonian vector field associated to $g$ and the initial hypersurface $S$, and second that a uniformizing map for the solution can be constructed explicitly through the solution of an auxiliary Cauchy problem for a Hamilton-Jacobi equation associated to the principal part $G$ of $A$ and the initial hypersurface $S$. We remark here that for the application of Leray's results to the local isometric embedding problem of the class of singular metrics considered in this paper, we shall be specifically interested in the case on which the data are characteristic at a single point $x\in S$, with an important non-degeneracy condition which will be specified below in Definition \ref{nondeg}.
 
The first step in Leray's construction is to extend the Cauchy problem
(\ref{scalarop}),(\ref{scalardata}) by the addition of an auxiliary
variable $\xi$ to the original set $x=(x_{1}\,\ldots,x_{l})$ of
independent variables. More precisely, we take an interval $(-\eta,\eta)$ and consider on $(-\eta,\eta)\times V \subset \mathbb{R}^{l+1}$ with coordinates $(\xi,x)$ the modified Cauchy problem given by
\begin{equation}\label{scalaropxi}
a(x,\partial_x)u(\xi,x)=v(\xi,x)\,,
\end{equation}
where the initial hypersurface $S$ is now replaced by the hypersurface
$S_{\xi}$ defined as the level set $s(x)=\xi$, where $v$ is assumed to
be analytic in $(-\eta,\eta)\times V$, and where the Cauchy problem is now given by prescribing the values of an analytic function $w(\xi,x)$ and its derivatives of order $0 \leq k\leq m-1$ on $S_{\xi}$,
\begin{equation}\label{scalardataxi}
u(\xi,x)-w(\xi,x)=O(s(\xi,x)^{m})\,,
\end{equation}
where $s(\xi,x):=s(x)-\xi$. The set of characteristic points for the Cauchy problem (\ref{scalaropxi}), (\ref{scalardataxi}) will be denoted by $C_{\xi}$.

Referring back to the original Cauchy problem (\ref{scalarop}), (\ref{scalardata}) for the operator $A$ defined by (\ref{DefA}), the key idea behind the uniformization of the solution near a characteristic point is to pass to the modified Cauchy problem (\ref{scalaropxi}), (\ref{scalardataxi}) and to construct the uniformizing map by means of the solution $\xi(t,x)$ of a suitably chosen auxiliary Cauchy problem, namely the following Cauchy problem for the Hamilton-Jacobi equation associated to the Hamiltonian $g$ defined by (\ref{Hamiltonian}): 
\begin{equation}\label{HJ}
\partial_t\xi+g(x,\partial_x\xi)=0\,,\quad \xi(0,x)=s(x)\,,
\end{equation} 
This Cauchy problem is solved as usual by the method of characteristic strips, thus giving rise to a unique analytic solution $\xi(t,x)$ defined for $|t| < \epsilon$ sufficiently small. Now by construction, the map $f: (-\epsilon,\epsilon)\times V \to (-\eta,\eta)\times V$ given by
\begin{equation}
(t,x)\mapsto f(t,x)=(\xi(t,x),x)\,,
\end{equation}
will map the hypersurface $t=0$ to the hypersurface $S_{\xi}$. The Jacobian determinant of this map is equal to $\partial_t\xi$, and the zero set of this determinant,  
\begin{equation}\label{zeroset}
Z_{\xi}:= \{(t,x)\in (-\epsilon,\epsilon)\times V\,\,|\,\, \partial_t\xi(t,x)=0\}\,,
\end{equation}
corresponds precisely by (\ref{HJ}) to the analytic subvariety of $(-\epsilon,\epsilon)\times V$ on which the characteristic condition $g(x,\partial_x\xi(t,x))=0$ is satisfied. This leads one naturally to define the characteristic conoid $K_{\xi}\subset (-\eta,\eta)\times V$ as the image under $f$ of $Z_{\xi}\subset  (-\epsilon,\epsilon)\times V$,
\begin{equation}\label{conoid}
K_{\xi}=f(Z_{\xi})\,.
\end{equation}
We thus see that for any analytic function $u(\xi,x)$, the map $(u \circ f)(t,x)=u(\xi(t,x),x)$ obtained by composition of $f$ with $u$ will be in general multivalued, and ramified precisely along the characteristic conoid $K_{\xi}$. 

It is useful at this stage to present the suggestive geometric description given by Leray of the characteristic conoid $K_{\xi}$ in terms of the characteristic strips for the Hamilton-Jacobi equation (\ref{HJ}). Indeed, using the canonical symplectic form $\omega$ on $T^{*}V$, 
\[
\omega=\sum_{i=1}^{l}dx_{i}\wedge dp_{i}\,,
\]
to define the Hamiltonian vector field $X_{g}$ on $T^{*}V$ corresponding to $g$, the characteristic strips associated to the Cauchy problem (\ref{HJ}) for the Hamilton-Jacobi equation are none but the integral curves $\gamma(t),\, -\epsilon <t< \epsilon$ of $X_{g}$ whose initial points $\gamma(0)$ are elements of the submanifold of $T^{*}V$ as a Lagrangian lift of $S$. In other words the characteristic strips $\gamma(t) = (x(t),y(t))$ are the solutions of the Cauchy problem given by
\begin{equation}\label{Ham}
\frac{dx}{dt}=\partial_{p}g\,,\quad \frac{dp}{dt}=-\partial_{x}g\,,\quad x(0,y)=y\,,\,p(0,y)=\partial_{y}s(y)\,,
\end{equation}
for the Hamiltonian system of ODEs associated $g$. The projections onto $V$ of the solution curves $\gamma(t),\, -\epsilon <t< \epsilon,$ of the Cauchy problem (\ref{Ham}) are by definition the bi-characteristic curves of the original Cauchy problem (\ref{scalarop}),(\ref{scalardata}). Leray proves in \cite{L} (see also \cite{G}) the following result, which gives the a geometric picture of the characteristic conoid $K_{\xi}$:
\begin{lemma}
The characteristic conoid $K_{\xi}$ is the union of the images of the bi-characteristic curves $x(t), \epsilon <t< \epsilon,$ whose initial points $x(0)$ are elements of the subset $C_{\xi}$ of characteristic points of $S_{\xi}$.  
\end{lemma}
If we fix a point $x\in C_{\xi}$, the subset of $K_{\xi}$ consisting of the images of the bi-characteristics $x(t)$, $-\epsilon <t< \epsilon $ such that $x(0)=x$ will be denoted by $K_{x}$ and referred to as the conoid with vertex at $x$. Our analyticity hypotheses on the coefficients $a_{\alpha}$ of $A$ imply immediately that $K_{x}\setminus \{x\}$ is an analytic submanifold of $V$. 

Finally, we need to define what is meant by an exceptional characteristic point for the Cauchy problem (\ref{scalaropxi}), (\ref{scalardataxi}):
\begin{defi}\label{nondeg}
A characteristic point $x\in C_{\xi}$ is said to be exceptional if either the initial hypersurface $S_{\xi}$ and the conoid $K_{x}$ are tangent to each other at infinitely many points in a neighbourhood of $x$, or the characteristic strip $\gamma(t),\,  -\epsilon <t< \epsilon,$ with initial point $\gamma(0)=(x,\partial_xs(x))$ consists of a single point.
\end{defi}

In our application of Leray's theory to the local isometric embeddings for the Riemannian metrics admitting an admissible singularity at the origin in the sense of Definition~1.1, the characteristic subset will consist of a single point, which will be a non-exceptional characteristic point in the sense of Definition \ref{nondeg} for the differential system governing local isometric embeddings.

We are now ready to state Leray's uniformization theorem:
\begin{thm}\label{Leraythm}
Let $\xi=\xi(t,x)$ be the solution of the Cauchy problem (\ref{HJ})
for the Hamilton-Jacobi equation (\ref{HJ}). In a neighborhood of a
non-exceptional characteristic point, the map $(t,x)\mapsto (\xi(t,x),x)$ uniformizes the solution $u(\xi,x)$ of the Cauchy problem (\ref{scalaropxi}), (\ref{scalardataxi}), in the sense that the composition
\[
u(\xi(t,x),x):=u\circ \xi\,,
\]
and its derivatives of order $1 \leq j\leq m-1$, 
\[
\partial_{\xi}^{j}u(\xi(t,x),x)\,,\quad 1 \leq j\leq m-1\,,
\]
are analytic for $(t,x)\in (-\epsilon,\epsilon)\times V$. Furthermore, the support of the ramification locus of the multi-valued function $u(\xi,x)$ solving the Cauchy problem (\ref{scalaropxi}), (\ref{scalardataxi}) lies in the set of points $(\xi,x)\in (-\eta,\eta)\times V$ for which the hypersurface $S_{\xi}$ is tangent to the conoid $K_{x}$. Likewise, the restriction of $u(\xi,x)$ to the locus $\xi(t,x)=0$ uniformizes the solution $u(0,x)$ of the original Cauchy problem (\ref{scalarop}), (\ref{scalardata}). Finally, the singularities of $u$ in the neighbourhood of any non-exceptional characteristic point are algebroid.
\end{thm}
We briefly illustrate Theorem \ref{Leraythm} on a very simple example taken from \cite{G}. Consider the Cauchy problem (\ref{scalarop}), (\ref{scalardata}) for the operator
\[
a(x,\partial_x):=\partial_{x_{1}}\,,
\]
with initial data given on the hypersurface $S$ defined by 
\[
x_{2}-x_{1}^{p}=0\,,
\]
where $p>0$ is a positive integer. The solution of this Cauchy problem is given by
\[
u(x)=w(x_{2}^{{1}/{p}},x_{2}\ldots,x_{l})+\int_{x_{2}^{{1}/{p}}}^{x_{1}}v(s,x_{2}\ldots,x_{l})\,ds\,.
\] 
This solution is not analytic, but it is ramified along the hyperplane $x_{2}=0$, with $p$ branches. The ramification locus $x_{2}=0$ is tangent to $S$ (to order $p$) along the characteristic submanifold of $S$ given by $x_{1}=x_{2}=0$, and the uniformization map is simply given by 
\[(x_{1},x_{2},\ldots,x_{l})\mapsto (x_{1},t^{p},\ldots,x_{l})\,.
\] 
For reasons of space, the summary we have presented of Leray's results on the uniformization of the Cauchy problem is both limited in scope and lacking in significant illustrative examples. We refer the reader to the paper by Johnsson \cite{John} for the application of Leray's uniformization theory to a variety of Cauchy problems for second-order linear operators, with data on a quadric surface. This paper contains many explicit examples.   

We conclude this appendix by briefly describing Choquet-Bruhat's generalization of Leray's uniformization theorem to the case of analytic systems, which is proved in \cite{CB}. In this setting, the scalar linear PDE (\ref{scalaropxi}) is replaced a non-linear system of $N$ PDEs for $N$ unknowns, of the form 
\begin{equation}\label{nonlinsys}
F[u]:=\big(F_j(x,\xi,D^{m}u)\big)_{j=1}^N=0\,,
\end{equation}
where $u(\xi,x)=(u_{1}(\xi,x),\ldots,u_{N}(\xi,x))$, where
$D^{m}u=(u,\nabla u,\dots,\nabla^m u)$ stands for the set of
derivatives up to order $ m$ of $u$ with respect to $x=(x_{1},\dots,x_{l})$, and where $F$ is analytic with respect to all its arguments. As shown in \cite{CB}, there is no loss of generality in expressing the system (\ref{nonlinsys}) in the form 
\begin{equation}\label{nonlinsysmod}
F[u]:=\big(F_j(x,\xi,D^{{m}_{k}-n_{j}}u)\big)_{j=1}^N=0\,,
\end{equation}
in which $m_{k}, n_{j}$ (with $1\leq k,j \leq N$) are non-negative integers through which the highest-order derivative appearing in each equation is made explicit. The Cauchy problem (\ref{scalardataxi}) is then replaced by the prescription of  $N$ analytic functions $w_k(\xi,x),\,1\leq k \leq N,$ and its derivatives of order $0 \leq k\leq m_{k}-1$ on $S_{\xi}$, that is
\begin{equation}\label{dataxisyst}
u_k(\xi,x)-w_k(\xi,x)=O(s(\xi,x)^{m_{k}})\,.
\end{equation}

As we mentioned earlier, the first step taken in \cite{CB} is to put
the system (\ref{nonlinsys}) in quasi-linear form by
differentiation. The vanishing condition (\ref{char}) that defines
characteristic points in the scalar linear case is then replaced by
the vanishing condition of the determinant of a matrix $\cal A$
governing the linear dependence of the highest-order derivatives
appearing in each of the differentiated equations. To state the
result, let us first define the matrix $\mathcal A(\xi,x,p)$ of components
\[
{\mathcal{A}}_{jk}(\xi,x,p):=\sum_{|\alpha|={m_{k}-n_{j}}} \frac{\partial
    F_j}{\partial(\partial^\alpha u_k)}\, p^\alpha\,.
\]

\begin{defi}
We say that $x\in S$ is characteristic for the Cauchy problem
(\ref{nonlinsysmod}), (\ref{dataxisyst}) if
${\mathcal{A}}_*(x)=0$, where
\begin{equation}\label{nonlinchar}
 {\mathcal{A}}_*(x):=\det\big(\mathcal A(s(x),x,\partial_xs(x))\big|_{u_k(x)=w_k(x)}\big)=0\,.
\end{equation}
\end{defi}

\begin{rem}\label{R.nonex}
  A sufficient condition ensuring that a characteristic
  point $x\in S$ is non-exceptional is that the existence of a tangent
  vector $Y\in T_xS$ such that the derivative of ${\mathcal{A}}_*$
  along~$X$ at~$x$ is nonzero.
\end{rem}

The rest of the analysis in \cite{CB} proceeds along lines that are
essentially similar to the ones we have described above for the scalar
linear case, with the function induced by $\det \mathcal{A}(s(x),x,p)$
on the cotangent bundle $T^{*}V$ playing the role of the Hamiltonian
$g(x,p)$. The statement of Theorem \ref{Leraythm} then carries over to the case of non-linear systems with the slight modification that it is now each of the components $u_k(\xi,x)$ of $u$ with its derivatives of order $1 \leq j\leq m_{k}-1$ which gets uniformized by the map $(t,x)\mapsto (\xi(t,x)),x)$:

\begin{thm}\label{T.CB}
  In a neighborhood of a non-exceptional characteristic point, the compositions
\begin{equation}
u_k(\xi(t,x),x):=u_k\circ \xi\,,\quad 1\leq k\leq N\,,
\end{equation}
and their derivatives of order $1 \leq j\leq m_{k}-1$, 
\[
\partial_{\xi}^{j}u_k(\xi(t,x),x)\,,\quad 1\leq k\leq N\,,1 \leq j\leq m_{k}-1\,.
\]
are analytic for $(t,x)\in (-\epsilon,\epsilon)\times V$. The support of the ramification locus admits the same description as in Theorem \ref{Leraythm}. In particular, the restriction of $u(\xi,x)$ to the locus $\xi(t,x)=0$ will uniformize the solution $u(0,x)$ of the Cauchy problem given by (\ref{nonlinsysmod}) and (\ref{dataxisyst}) in the precise sense of Theorem \ref{Leraythm}. Likewise the singularities of $u$ in the neighbourhood of any non-exceptional characteristic point will be algebroid.

\end{thm}
\section*{Acknowledgements}
A.E.'s research is supported supported by the grants ERC StG 633152, SEV-2015-0554 and 20205CEX001. N.K.'s research is supported by NSERC grant RGPIN 105490-2018. One of us (N.K.) would like to thank the ICMAT for its warm hospitality during the research visits that led to the completion of this paper.

\end{document}